\documentclass[11pt]{article}
\usepackage{color,graphicx}

\usepackage{amsmath,amsfonts,amssymb,amsthm,placeins,float}

\begin{document}

\title{Maximal nonassociativity via fields}
\author{Petr Lison\v{e}k\thanks{Research was supported
in part by the Natural Sciences and Engineering Research Council of Canada
(NSERC).}
\\
Department of Mathematics\\
Simon Fraser University\\
Burnaby, BC\\
Canada\ \ \ V5A 1S6\\
\ \\
{\tt plisonek@sfu.ca}}
\date{}

\def\eps{\varepsilon}
\def\F{{\mathbb F}}
\def\Q{{\mathbb Q}}
\def\C{{\mathbb C}}
\def\Z{{\mathbb Z}}
\def\Fq{{\mathbb F_q}}
\def\e{\eta}
\def\L1{L(1,z^*)}
\def\Le{L(\e,z^*)}

\def\P{\phantom{\Big|}}

\newtheorem{theorem}{Theorem}[section]
\newtheorem{lemma}[theorem]{Lemma}
\newtheorem{proposition}[theorem]{Proposition}
\newtheorem{corollary}[theorem]{Corollary}
\newtheorem{conjecture}[theorem]{Conjecture}
\theoremstyle{definition}
\newtheorem{definition}[theorem]{Definition}
\newtheorem{example}[theorem]{Example}
\newtheorem{remark}[theorem]{Remark}

\def\maqs{maximally nonassociative quasigroups}
\def\maq{maximally nonassociative quasigroup}
\def\ra{\rightarrow}

\def\sq{square}
\def\ns{non-square}

\maketitle

\begin{abstract}
We say that $(x,y,z)\in Q^3$ is an associative triple in a quasigroup $Q(*)$
if $(x*y)*z=x*(y*z)$. Let $a(Q)$ denote the number of associative triples in $Q$.
It is easy to show that $a(Q)\ge |Q|$, and we call the quasigroup 
maximally nonassociative if $a(Q)= |Q|$. It was conjectured 
that maximally nonassociative quasigroups do not exist when $|Q|>1$.
Dr\'apal and Lison\v{e}k recently refuted this conjecture
by proving the existence of maximally nonassociative quasigroups
for a certain infinite set of orders $|Q|$.
In this paper we prove the existence of maximally nonassociative quasigroups
for a much larger set of orders $|Q|$. Our main tools are finite fields
and the Weil bound on quadratic character sums.
Unlike in the previous work, our results are to a large extent constructive.
\end{abstract}

\section{Maximally nonassociative quasigroups}

A {\em quasigroup} $Q(*)$ 
is a set $Q$ with a binary operation $*$ such that
for all $a,b\in Q$ there exist unique $x,y \in Q$ such that 
$a*x = b $ and $y * a = b$. 
Hence a binary 
operation on a finite set yields a quasigroup if and only if 
its multiplication table is a Latin square.

We call a triple $(x,y,z)\in Q^3$  \emph{associative} if 
$(x*y)*z = x*(y*z)$,
and we denote the number of associative triples in $Q$ by $a(Q)$.
For each $c\in Q$ there exist $x,y$ such that $c*x=c$ and $y*c=c$.
Then $(y,c,x)$ is an associative triple and it follows
that $a(Q)\ge |Q|$.
We call $Q$ {\em maximally nonassociative} when $a(Q)= |Q|$.
The existence of such quasigroups has been investigated
for at least four decades \cite{tk}, \cite{kr}.
It was showed \cite{gh} that
quasigroups with few associative triples
can be used in the design of hash functions in cryptography.
Gro\v sek and Hor\'ak 
conjectured that  \maqs\ do not exist when $|Q|>1$
\cite[Conjecture~1.2]{gh}.
An example of a~\maq\ of order~9 was found recently \cite{dv2}.
No example of order greater than~$1$
and less than~$9$ exists,
and no example of order~$10$ exists \cite{dv1,dv2}.
It is known 
\cite[Theorem~1.1]{gh}
that any maximally nonassociative quasigroup must
be {\em idempotent,} that is, $x*x=x$ for all $x\in Q$.

Very recently Dr\'apal and Lison\v{e}k \cite{DL}
used Dickson's quadratic nearfields to prove
existence of an infinite set of \maqs.
Specifically they proved:

\begin{theorem}\cite[Corollary~5.8]{DL}
Let $m=2^{3k}r$ where $k\ge 0$ is an integer and $r$ is odd.
There exists a maximally nonassociative quasigroup
of order $m^2$.
\end{theorem}

In this paper we greatly  extend
the set of orders for which the existence
of \maqs\ can be proved (see Theorem~\ref{thm-main} below).
As well, our approach is more constructive
in comparison to the methods of~\cite{DL}.

\subsection{Prescribed automorphism groups}

The method of prescribed automorphism group has been widely
used to construct combinatorial objects with distinguished properties.
One starts with some permutation group $G$, and with the assumption
that $G$ is a group of symmetries of the desired object
(but perhaps not its full automorphism group).
It is then sufficient to specify the behaviour of the object
on the orbits of $G$, thus cutting down the complexity
of computer search or theoretical proofs.
Often there is some natural choice for $G$, which may or may not work.
If $G$ itself does not work,
then one uses proper subgroups of $G$.
This process will proceed by increasing the index of the subgroup,
from smaller index to larger index. The reason for proceeding
in this particular way is that the number of orbits (and thus
the complexity of computer search or theoretical proofs)
will generally increase
with the index of the subgroup.

Let $(Q,*)$ be a quasigroup
and let $L$ be its multiplication table.
Then $L$ is a Latin square with symbol set~$Q$.
Rows and columns of $L$ are also indexed by $Q$.
Letting $L(x,y)$ denote
the entry in row $x$ and column $y$ of~$L$, we have $L(x,y)=x*y$.
This relation couples $L$ and $Q$,
and it makes sense to speak of ``quasigroup $L$'' as well.

An automorphism of $L$ is a bijection $f: Q\rightarrow Q$
such that $L(f(x),f(y))=f(L(x,y))$
for all $x,y\in Q$.
We note that $(x,y,z)$ is associative
if and only if $(f(x),f(y),f(z))$ is associative.
Note that the definition of automorphism makes sense
even when the mapping $L$ does not possess the Latin property
(i.e., $L$~is not a quasigroup).
We will use this broader definition of automorphism in Lemma~\ref{Lab-G}
below.

We wish to apply the idea of prescribed
automorphism group $G$ to constructing \maq\ $L$.
Since $L$ is known to be idempotent,
we need to specify only the values $L(x,y)$ for $x\neq y$.
The most desirable choice  for $G$ would be a group
that acts primitively transitively on the ordered pairs $(x,y)$
with $x\neq y$, known as {\em sharply 2-transitive (S2T) group.}
It is known that finite S2T groups are in one-to-one correspondence with
{\em nearfields,}  and this underlies the approach taken in~\cite{DL},
where \maqs\ are obtained from {\em proper} nearfields,
that is, those which are not fields.

It still remains to explore the case when the nearfield
is not proper, that is, it is the finite field of order~$q$,
denoted $\F_q$.
In this case the corresponding S2T group
is the Frobenius group of invertible
affine mappings $x\mapsto \alpha x+\beta$ on $\F_q$.
Let $G_q$ denote this group.
It appears that $G_q$ itself can not be used to produce \maqs\ \cite{DL}.
Therefore, following the strategy outlined above,
one proceeds to proper subgroups of $G_q$ of low index.
Luckily, when $q$ is odd, there is always
a subgroup of index~$2$, which is the group of invertible
affine mappings $x\mapsto \alpha x+\beta$ 
where $\alpha$ is a non-zero square in $\F_q$.
Amazingly, this group works for constructing 
a very large set of \maqs, and this is the approach that we take
in this paper.

\section{Constructions}

% We now proceed to the proof that \maqs\ exist for a wide range of orders.

By the {\em parity} of $a\in\Fq$ we mean the squareness of~$a$,
hence the parity is either ``square'' or ``\ns.''

\begin{definition}
\label{def-Lab}
Let $q$ be an odd prime power.
For fixed $a,b\in\Fq$ define the mapping $L_{a,b}:\Fq\times\Fq\ra\Fq$ by
\[
L_{a,b}(x,y)=
\begin{cases}
x+a(y-x) &\mbox{if\ $y-x$\ is\ \sq}\\
x+b(y-x) &\mbox{if\ $y-x$\ is\ \ns}.
\end{cases}
\]
\end{definition}
Note that $L_{a,b}(x,x)=x$ for all $x\in\Fq$.
In the special case $a=b\neq 0,1$ we get a classical
construction of idempotent Latin squares due
to Bose, Parker and Shrikhande \cite[p.~288]{vLW}.

\begin{lemma}
\label{Lab-G}
Let $q$ be an odd prime power
and let $L_{a,b}:\Fq\times\Fq\ra\Fq$ be as above.
%in Definition~\ref{def-Lab}.
Let $G_q^{(2)}$ denote the group of mappings $u\mapsto \alpha u+\beta$
where $\alpha,\beta\in\Fq$ and $\alpha$ is a non-zero square.
Then $G_q^{(2)}$ is an automorphism group of $L_{a,b}$.
\begin{proof}
Let $f\in G_q^{(2)}$,
$f(u)= \alpha u+\beta$.
Let $(x,y)\in\F_q^2$ and assume $y-x$ is square. Then $f(y)-f(x)=\alpha(y-x)$
is also a square, and
\begin{eqnarray*}
f(L_{a,b}(x,y))
&=&f(x+a(y-x))
=\alpha(x+a(y-x))+\beta\\
&=&
\alpha x+\beta + a(\alpha(y-x))
=
f(x)+a(f(y)-f(x))\\
&=&L_{a,b}(f(x),f(y)).
\end{eqnarray*}
The case when $y-x$ is \ns\ is analogous.
\end{proof}
\end{lemma}

Letting $G_q^{(2)}$ act naturally on $\F_q^2$
we observe that there are three orbits. One orbit
contains the diagonal pairs $(x,x)$,
and  $G_q^{(2)}$ acts
primitively on the remaining two orbits,
one of which contains pairs $(x,y)$ where $y-x$ is non-zero square
and 
the other orbit contains pairs where $y-x$ is \ns.

\begin{proposition}
\label{prop-LS}
Let $q$ be an odd prime power 
and let $a\in\Fq$ be such that
$a\not\in\{-1,0,1\}$ and 
both $a$ and $a+1$ are squares. Then the mapping $L_{a,a^2}$
is a Latin square.
\end{proposition}
\begin{proof}
Assume that $L_{a,a^2}(x,y_1)=L_{a,a^2}(x,y_2)$.
If $y_1-x$ and $y_2-x$ have the same parity, then
$x+a^k(y_1-x)=x+a^k(y_2-x)$ for some $k\in\{1,2\}$ and $y_1=y_2$
since $a\neq 0$.
If $y_1-x$ and $y_2-x$ have opposite parities, then
without loss of generality assume that
$y_1-x$ is \sq\ and $y_2-x$ is \ns.
Since $a$ is square, we get that $a(y_1-x)$ is \sq\
but $a^2(y_2-x)$ is \ns, which is a contradiction.

Now assume that $L_{a,a^2}(x_1,y)=L_{a,a^2}(x_2,y)$.
If $y-x_1$ and $y-x_2$ have the same parity, then
$x_1+a^k(y-x_1)=x_2+a^k(y-x_2)$ for some $k\in\{1,2\}$.
Then $x_1(1-a^k)=x_2(1-a^k)$ and $x_1=x_2$ since $a\neq -1,1$.
Finally assume that 
$y-x_1$ and $y-x_2$ have opposite parities. Without loss
of generality assume that $s=y-x_1$ is \sq\ and $n=y-x_2$ is \ns,
then we get
\begin{eqnarray*}
x_1+a(y-x_1) & = & x_2+a^2(y-x_2) \\
y-s+as & = & y-n+a^2n \\
s(a-1) & = & n(a^2-1) \\
s & = & n(a+1).
\end{eqnarray*}
Since $a+1$ is assumed to be \sq\
and $a\neq -1$, the right-hand side of the last equation
is \ns\ while its left-hand side is a \sq, a contradiction.
\end{proof}

\begin{theorem}
\label{thm-q14}
Let $q\equiv 1\pmod{4}$ be a prime power.
Let $a\in \F_q$ be such that
$a\neq -1,0,1$, 
and the elements $a$, $a+1$, $a^3-a-1$ are \sq s,
and the elements $a-1$, $a^2+1$, $a^2-a-1$, $a^2+a+1$, $a^2+a-1$ are \ns s.
Then $L_{a,a^2}$ is a \maq.
\end{theorem}
\begin{proof}
Since the assumptions of Proposition~\ref{prop-LS}
are a subset of the assumptions of Theorem~\ref{thm-q14},
it follows that $L_{a,a^2}$ as a Latin square.
For simplicity let us abbreviate $L_{a,a^2}$ as $L$.

Since $q\equiv 1\pmod{4}$, we note that $-1$ is a \sq,
which will be used often throughout this proof.
Let $\eta\in\Fq$ be a fixed \ns.

First we show that there are
no associative triples of the form $(t,t,v)$, where $t\neq v$.
For any such triple there exists an automorphism $f$
of $L_{a,a^2}$ such that 
$(f(t),f(t),f(v))=(0,0,z)$ for some $z\in\Fq$, $z\neq 0$.
Towards a contradiction assume that 
$L(L(0,0),z)=L(0,L(0,z))$. This simplifies
to $L(0,z)=L(0,L(0,z))$. Let $\ell=L(0,z)$.
From Definition~\ref{def-Lab}
we get $\ell=a\ell$ or $\ell=a^2\ell$.
Since $a\neq \pm 1$, we get $L(0,z)=0$.
Since $L(0,0)=0$ and $L$ is a Latin square,
we get $z=0$, a contradiction.

Now let $(t,u,v)\in\F_q^3$ be an associative triple
such that $t\neq u$.  
If $u-t$ is \sq,
then there exists an automorphism $f$ of $L_{a,a^2}$
such that $(f(t),f(u),f(v))=(0,1,z)$
for some $z\in\Fq$.
If $u-t$ is \ns,
then there exists an automorphism $f$ of $L_{a,a^2}$
such that $(f(t),f(u),f(v))=(0,\eta,z)$
for some $z\in\Fq$.
Hence it is sufficient to prove the nonexistence
of associative triples of the form $(0,1,z)$ and $(0,\eta,z)$.

Suppose that $(0,1,z)$ is an associative triple,
that is, $L(L(0,1),z)=L(0,L(1,z))$.
Since $1$ is \sq, we get $L(0,1)=0+a(1-0)=a$
and the associative triple condition simplifies to
\begin{equation}
\label{eq-01z}
L(a,z)=L(0,L(1,z)).
\end{equation}
The three values of the function $L$ seen in equation (\ref{eq-01z})
depend on the parities of the three elements 
\begin{equation}
\label{eq-c1-c2-c3}
c_1=z-a, \ \ \ c_2=z-1, \ \ \ c_3=L(1,z).
\end{equation}
Thus there are eight cases to consider. In each case
the parities of $c_1$, $c_2$ and $c_3$ are fixed,
and equation (\ref{eq-01z}) can be written down explicitly
in terms of $a$ and $z$ and no other variables or functions.
This equation may be free of~$z$, in which case it turns out
to be never satisfied due to the assumptions on~$a$,
or the equation is {\em linear} in $z$
and it has a unique root in $\Fq$ which we denote $z^*$.
The contradiction is then obtained by showing that
at least one of the elements 
$z^*-a$, $z^*-1$, $L(1,z^*)$ has the {\em opposite} parity
in comparison to
what was originally assumed, hence showing
that this case can not occur. Once all eight cases are shown
to be impossible, it follows that there is no $z\in\Fq$
such that $(0,1,z)$ is an associative triple.

\begin{table}[H]
\begin{center}
\begin{tabular}{lllll}
\hline
$z-a$  \ \ &$z-1$  \ \ &$L(1,z)$  \ \ \ \ \ \ &$z^*$\ \ \ \ \ \  \ \ \ \ \ \ &contradiction\\
\hline
%1
S &S &S   &$0 $    &$\L1=1-a  \P$ is N     \\
%2
S &S &N   &$\frac{a-1}{a+1}  \P$    &$z^*-a=-\frac{a^2+1}{a+1} $ is N      \\
%3
S &N &S   &$\frac{a}{a+1}  \P$    &$z^*-1=-\frac{1}{a+1} $ is S      \\
%4
S &N &N   &$\frac{a^2+a-1}{a^2+a+1} \P $    &$\L1=-\frac{a^2-a-1}{a^2+a+1} $ is S      \\
%5
N &S &S   &$\rm{none}  \P$    &$  $      \\
%6
N &S &N &$-\frac{1}{a}   \P$ &$\L1=-a  $ is S\\ 
%7
N &N &S   &$ 0 \P$    &$z^*-a=-a  $ is S      \\
%8
N &N &N   &$\frac{a-1}{a}  \P$    &$z^*-1=-\frac{1}{a} $ is S     \\
\hline
\end{tabular}
\end{center}
\caption{Non-existence of associative triples $(0,1,z)$ when $q\equiv 1\pmod{4}$.}
\label{tab-1-1}
\end{table}

\begin{table}[H]
\begin{center}
\begin{tabular}{lllll}
%d_1=z-a^2\e, \ \ \ d_2=z-\e, \ \ \ d_3=L(\e,z).
\hline
$z-a^2\e$  \ \ &$z-\e$  \ \ &$L(\e,z)$  \ \ \ \ \ \ &$z^*$\ \ \ \ \ \  \ \ \ \ \ \ &contradiction\\
\hline
%9
S &S &S   &$-\e(a-1)  \P$    &$z^*-\e=-a\e $ is N      \\
%10
S &S &N   &$0  \P$    &$z^*-\e=-\e $ is N      \\
%11
S &N &S   &$\frac{\e}{a+1}  \P$    &$\Le=-\frac{\e(a^3-a-1)}{a+1} $ is N     \\
%12
S &N &N   &$\frac{a^2\e}{a^2+a+1}  \P$    
   &$z^*-\e=-\frac{\e(a+1)}{a^2+a+1} $ is S       \\
%13
N &S &S   &$\rm{none}  \P$    &$ $      \\
%14
N &S &N &$-a\e \P$ &$z^*-\e=-(a+1)\e $ is N\\ 
%15
N &N &S   &$-\frac{(a^2-1)\e}{a}  \P$    &$z^*-\e=-\frac{\e(a^2+a-1)}{a} $ is S     \\
%16
N &N &N   &$ 0 \P$    &$\Le=-\e(a-1)(a+1) $ is S      \\
\hline
\end{tabular}
\end{center}
\caption{Non-existence of associative triples $(0,\eta,z)$ when $q\equiv 1\pmod{4}$.}
\label{tab-1-e}
\end{table}

We show all details of the proof for two of the eight cases;
the other cases are similar. 
The details of all eight cases are summarized in Table~\ref{tab-1-1},
where S denotes \sq\ and N denotes \ns.

% Case 5
First assume that $c_1=z-a$ is \ns\ 
and $c_2=z-1$, $c_3=L(1,z)$ are both \sq.
Applying Definition~\ref{def-Lab} 
we get $L(1,z)=1+a(z-1)$ and then 
applying Definition~\ref{def-Lab} 
to equation (\ref{eq-01z}) we get
\begin{eqnarray*}
L(a,z) & = & L(0,L(1,z)) \\
a+a^2(z-a) & = & 0 + a(1+a(z-1)   - 0)  \\
a^2(a-1) & = &  0 
\end{eqnarray*}
which is impossible since $a\neq 0,1$. This is a case where no $z^*$ exists.

% Case 6
Next assume that $c_1=z-a$ is \ns,
$c_2=z-1$ is \sq, and $c_3=L(1,z)$ is \ns.
Applying Definition~\ref{def-Lab} 
we get $L(1,z)=1+a(z-1)$ and then 
applying Definition~\ref{def-Lab} 
to equation (\ref{eq-01z}) we get
\begin{eqnarray*}
L(a,z) & = & L(0,L(1,z)) \\
a+a^2(z-a) & = & 0 + a^2(1+a(z-1)   - 0)  \\
a(a-1)(az+1) & = &  0.
\end{eqnarray*}
Since $a\neq 0,1$, the last equation
has the unique solution $z^*=-1/a$.
By assumption, any $z$ occurring in this subcase
satisfies that $z-1$ is \sq,
hence we can evaluate 
\[
L(1,z^*)=1+a(z^*-1)=1+a(-1/a-1)=1-1-a=-a.
\]
Since $-1$ is square when $q\equiv 1\pmod{4}$
and $a$ is assumed to be square,
it follows that $-a$, and hence also $L(1,z^*)$, is \sq.
This contradicts the assumption that $L(1,z)$ is \ns.

The remaining six cases are similar.
The details of all eight cases are recorded in Table~\ref{tab-1-1}.

In the second half of the proof
suppose that $(0,\e,z)$ is an associative triple,
that is, $L(L(0,\e),z)=L(0,L(\e,z))$.
Since $\e$ is \ns, we get $L(0,\e)=0+a^2(\e-0)=a^2\e$
and the associative triple condition simplifies to
\begin{equation}
\label{eq-0ez}
L(a^2\e,z)=L(0,L(\e,z)).
\end{equation}
The three values of the function $L$ seen in equation (\ref{eq-0ez})
now depend on the parities of the three elements 
\begin{equation}
\label{eq-d1-d2-d3}
d_1=z-a^2\e, \ \ \ d_2=z-\e, \ \ \ d_3=L(\e,z).
\end{equation}
Again we decompose this half of the proof into eight
cases that cover all possible combinations
of parities of $d_1,d_2,d_3$ and we drive each case to a contradiction.
The details are recorded in Table~\ref{tab-1-e}.

Note that all field elements recorded in Tables~\ref{tab-1-1}
and \ref{tab-1-e} exist, that is, there is no division by zero.
The denominators occurring in the tables
are $a$, $a+1$ and $a^2+a+1$. The first two can not be zero because $a\neq 0,-1$,
and the last one can not be zero because $a^2+a+1$ is assumed to be \ns.
This completes the proof of the theorem.
\end{proof}

\begin{theorem}
\label{thm-q34}
Let $q\equiv 3\pmod{4}$ be a prime power.
Let $a\in \F_q$ be such that
$a\neq -1,0,1$, 
and the elements $a$, $a+1$, $a-1$, $a^2+1$, $a^3+a^2-1$  are \sq s,
and the elements $a^2-a+1$, $a^2+a+1$, $a^2+a-1$  are \ns s.
Then $L_{a,a^2}$ is a \maq.
\end{theorem}
\begin{proof}
The structure of the proof is very similar to the case $q\equiv 1\pmod{4}$.
The reason why we need a separate proof for $q\equiv 3\pmod{4}$
is, of course, that $-1$ is now \ns.

As in the previous proof we first note that
$L_{a,a^2}$ is a Latin square, and again
we will abbreviate $L_{a,a^2}$ as $L$.
Again let $\eta\in\Fq$ denote a fixed \ns.
We could have achieved some simplifications
%in Table~\ref{tab-3-e} 
by taking $\eta=-1$,
however we chose to keep the symbol $\eta$ to illustrate
the full analogy with the case $q\equiv 1\pmod{4}$.

The proof
of non-existence of associative triples of the form $(t,t,v)$, where $t\neq v$,
which was given for $q\equiv 1\pmod{4}$ above, does not require any changes.
The rest of the proof is organized exactly as in Theorem~\ref{thm-q14}.

\begin{table}[H]
\begin{center}
\begin{tabular}{lllll}
\hline
$z-a$  \ \ &$z-1$  \ \ &$L(1,z)$  \ \ \ \ \ \ &$z^*$\ \ \ \ \ \  \ \ \ \ \ \ &contradiction\\
\hline
%1
S &S &S   &$0 \P$    &$z^*-1=-1 $ is N     \\
%2
S &S &N   &$\frac{a-1}{a+1}  \P$    &$z^*-a=-\frac{a^2+1}{a+1} $ is N      \\
%3
S &N &S   &$\frac{a}{a+1}  \P$    &$z^*-a=-\frac{a^2}{a+1} $ is N      \\
%4
S &N &N   &$\frac{a^2+a-1}{a^2+a+1} \P $    
     &$z^*-a=-\frac{(a+1)(a^2-a+1)}{a^2+a+1} $ is N     \\
%5
N &S &S   &$\rm{none}  \P$    &$  $      \\
%6
N &S &N &$-\frac{1}{a}   \P$ &$z^*-1=-\frac{a+1}{a}  $ is N\\ 
%7
N &N &S   &$ 0 \P$    &$\L1=-(a-1)(a+1)  $ is N      \\
%8
N &N &N   &$\frac{a-1}{a}  \P$    &$z^*-a=-\frac{a^2-a+1}{a} $ is S     \\
\hline
\end{tabular}
\end{center}
\caption{Non-existence of associative triples $(0,1,z)$ when $q\equiv 3\pmod{4}$.}
\label{tab-3-1}
\end{table}

\begin{table}[H]
\begin{center}
\begin{tabular}{lllll}
%d_1=z-a^2\e, \ \ \ d_2=z-\e, \ \ \ d_3=L(\e,z).
\hline
$z-a^2\e$  \ \ &$z-\e$  \ \ &$L(\e,z)$  \ \ \ \ \ \ &$z^*$\ \ \ \ \ \  \ \ \ \ \ \ &contradiction\\
\hline
%9
S &S &S   &$-\e(a-1)  \P$    &$z^*-a^2\e=-\e(a^2+a-1) $ is N      \\
%10
S &S &N   &$0  \P$    &$\Le=-\e(a-1) $ is S      \\
%11
S &N &S   &$\frac{\e}{a+1}  \P$    &$z^*-\e=-\frac{a\e}{a+1} $ is S     \\
%12
S &N &N   &$\frac{a^2\e}{a^2+a+1}  \P$    
   &$z^*-a^2\e=-\frac{a^3\e(a+1)}{a^2+a+1} $ is N       \\
%13
N &S &S   &$\rm{none}  \P$    &$ $      \\
%14
N &S &N &$-a\e \P$ &$z^*-a^2\e=-a\e(a+1) $ is S\\ 
%15
N &N &S   &$-\frac{(a^2-1)\e}{a}  \P$    
     &$z^*-a^2\e=-\frac{\e(a^3+a^2-1)}{a} $ is S     \\
%16
N &N &N   &$ 0 \P$    &$z^*-a^2\e=-a^2\e $ is S      \\
\hline
\end{tabular}
\end{center}
\caption{Non-existence of associative triples $(0,\eta,z)$ when $q\equiv 3\pmod{4}$.}
\label{tab-3-e}
\end{table}

Of course, the values of $z^*$ in all 16~cases are the same
as those found in the proof of Theorem~\ref{thm-q14}.
The details are recorded in Tables~\ref{tab-3-1} and \ref{tab-3-e}.

There is no division by zero in the tables for the same
reasons as in the case $q\equiv 1\pmod{4}$,
namely it is assumed that
$a\neq 0,-1$
and $a^2+a+1$ is \ns.
\end{proof}

% \FloatBarrier

\section{Applying the Weil bound}

We state the well known {\em Weil bound} 
on multiplicative character sums.

\begin{theorem}\cite[Theorem~6.2.2]{Evans}
\label{thm-weil}
Let $g\in\F_q[x]$ be a polynomial of degree $d>0$
and $\chi:\F_q^*\rightarrow\C^*$ a non-trivial multiplicative character
of order $m$ (extended by zero to $\F_q$). Then, if $g$ is not
an $m$-th power in $\overline{\F}_q[x]$ (where $\overline{\F}_q$ 
is the algebraic closure of $\F_q$),
\begin{equation}
\label{eq-weil-bound}
\left| \sum_{x\in\F_q} \chi(g(x)) \right| \le (d-1)\sqrt{q}.
\end{equation}
\end{theorem}

Let $q$ be an odd prime power and let $\chi$
be the quadratic character on $\F_q^*$,
that is, 
$\chi(u)=1$ if $u$ is a non-zero \sq\ in $\Fq$
and
$\chi(u)=-1$ if $u$ is a \ns\ in $\Fq$.
We extend $\chi$ on $\Fq$ by defining $\chi(0)=0$.

\begin{lemma}
\label{lem-kappa}
Let $f_1,\ldots,f_n$ be polynomials in $\Fq[x]$ and let
$1\le t< n$. Let 
$\eps_i=1$ for $1\le i\le t$
and
$\eps_i=-1$ for $t+1\le i\le n$.
Define $\kappa:\F_{q}\rightarrow \Q$ by
\begin{equation}
\label{eq-kappa}
\kappa(x)=
\frac{1}{2^n}\prod_{i=1}^n(1+\eps_i\chi(f_i(x)))
\end{equation}
and let $S=\sum_{x\in\Fq} \kappa(x)$.
The number of $a\in\Fq$ such that
$f_i(a)$ is \sq\ for $1\le i\le t$
and 
$f_i(a)$ is \ns\ for $t+1\le i\le n$
is at least $S-\sum_{i=1}^n \deg f_i$.
\end{lemma}
\begin{proof}
The result follows from three simple observations:
If $a\in\Fq$ is such that $f_i(a)\neq 0$ for $1\le i \le n$
and all $f_i(a)$ have the desired parities,
then $\kappa(a)=1$.
If $a\in\Fq$ is such that $f_i(a)\neq 0$ for $1\le i \le n$
and at least one  $f_i(a)$ has the wrong parity,
then $\kappa(a)=0$.
For each $a\in\Fq$ we have $\kappa(a)\le 1$,
and the total number of roots of the polynomials $f_i$ in $\Fq$
is at most $\sum_{i=1}^n \deg f_i$.
\end{proof}

%\begin{theorem}\cite[Theorem~5.6]{DL}
%Let $q$ be an odd prime power. There exists a \maq\ of order $q^2$.
%\end{theorem}

\begin{lemma}
\label{lem-24mil}
(i) For each odd prime power $q$ such that $q\le 7^3$
and\break 
$q\neq 3,5,7,11$ there exist $a,b\in\Fq$
such that $L_{a,b}$ is a \maq.\\
(ii) For each odd prime power $q$ such that $7^3 < q < 2.4\cdot 10^6$
and\break 
$q\neq 3^7,3^9,3^{11},3^{13}$ there exists $a\in\Fq$
such that $L_{a,a^2}$ is a \maq.
\end{lemma}
\begin{proof}
For each $q$ we found the required \maq\ using the computational
algebra system Magma \cite{mag}.
The total computation time was about 40~hours on a single CPU.
\end{proof}

The following result is well known and it is proved
by a direct product construction.

\begin{lemma}
\label{lem-dir-product}
Suppose that there exist \maq s of order $r$ and $s$,
then there exists a \maq\ of order $rs$.
\end{lemma}

Theorems \ref{thm-q1mod4} and \ref{thm-q3mod4} 
are
the main results of this section.

\begin{theorem}
\label{thm-q1mod4}
Let $q\equiv 1\pmod{4}$ be a prime power, $q\ge 9$.
There exists a \maq\ of order $q$.
\end{theorem}
\begin{proof}
The bulk of the proof is based on the combination of
Theorem~\ref{thm-q14}, Lemma~\ref{lem-kappa} and Theorem~\ref{thm-weil}.
Some special cases will remain to be handled by other methods,
and this will be done at the end of the proof.

To assume the notation of Lemma~\ref{lem-kappa},
let $f_1(x)=x$, $f_2(x)=x+1$, $f_3(x)=x^3-x-1$,
$f_4(x)=x-1$, $f_5(x)=x^2+1$, $f_6(x)=x^2-x-1$, $f_7(x)=x^2+x+1$,
$f_8(x)=x^2+x-1$, 
and let $t=3$,
that is,
$\eps_1=\eps_2=\eps_3=1$
and
$\eps_4=\eps_5=\eps_6=\eps_7=\eps_8=-1$.

Note that the quadratic character $\chi$ remains multiplicative
after extension at~0: we have $\chi(uv)=\chi(u)\chi(v)$
for all $u,v\in\Fq$ (not just on $\F_q^*$).
By expanding the right-hand side of (\ref{eq-kappa})
and rearranging the sums we get
\begin{eqnarray*}
\kappa(x)
-
\frac{1}{2^8}  
&=&
\frac{1}{2^8}  \sum_{\emptyset\neq I\subseteq\{1,\ldots,8\}}
                    \prod_{i\in I}(\eps_i\chi(f_i(x)))
\\
&=&
\frac{1}{2^8}  \sum_{\emptyset\neq I\subseteq\{1,\ldots,8\}}
                    \left( \prod_{i\in I} \eps_i \right)
                     \chi\left(\prod_{i\in I} f_i(x)\right).
\end{eqnarray*}
After summing over all $x\in\Fq$ and denoting $S=\sum_{x\in\Fq}\kappa(x)$
we get
\begin{equation}
\label{eq-S-char-sums}
S-\frac{q}{2^8}=
\frac{1}{2^8}  \sum_{\emptyset\neq I\subseteq\{1,\ldots,8\}}
                    \left( \prod_{i\in I} \eps_i \right)
                     \left(\sum_{x\in\Fq}\chi\left(\prod_{i\in I} f_i(x)\right)\right).
\end{equation}
We would like to apply the Weil bound 
(Theorem~\ref{thm-weil}) on the character sums
$\sum_{x\in\Fq}\chi\left(\prod_{i\in I} f_i(x)\right)$.
We exclude the cases where $\prod_{i\in I} f_i(x)$
may be a perfect square over some algebraic extension of~$\Fq$,
for some~$I$,
as follows. We know that the discriminant of a 
non-constant polynomial
which is a perfect square is equal to~$0$.
Consider a fixed index set~$I$.
We consider the polynomials
$f_i$ over $\Z$ 
and we compute the discriminant of $\prod_{i\in I} f_i(x)$,
let us say the discriminant is the integer~$D$.
Then the discriminant of $\prod_{i\in I} f_i(x)$
with $f_i$ over $\Fq$ equals $D\bmod p$, where $p$
is the characteristic of $\Fq$. In other words,
the prime factors of~$D$ are the characteristics
in which the Weil bound may not be applicable,
and we will deal with those characteristics separately at the end of the proof.

By computing the discriminant of $\prod_{i\in I} f_i(x)$ over $\Z$
for all non-empty subsets $I$ of $\{1,\ldots,8\}$
such that the degree of $\prod_{i\in I} f_i(x)$ is even
(which is a necessary condition for a polynomial to be a perfect square)
we find that the only primes that divide any of these discriminants
are $2$, $3$, $5$ and $23$. Only the last three are odd and relevant
to this theorem.

First assume that the characteristic of~$\Fq$ is different from 
$3$, $5$ and $23$.
Then the Weil bound applies to each character sum
on the right-hand side of (\ref{eq-S-char-sums}).
Applying Theorem~\ref{thm-weil} to each sum
$\sum_{x\in\Fq}\chi\left(\prod_{i\in I} f_i(x)\right)$ and
adding over all non-empty subsets of $\{1,\ldots,8\}$
we get
\begin{equation}
\label{eq-S-bound-numeric}
S\ge
\frac{1}{2^8}(q-1537\sqrt{q})
\end{equation}
where $1537=2^7(1+1+1+2+2+2+2+3)-2^8+1$
is the sum of $\deg(\prod_{i\in I} f_i(x))-1$ over all 
$\emptyset\neq I\subseteq\{1,\ldots,8\}$.
It is verified easily that for $q>2.4\cdot 10^6$
we have $\frac{1}{2^8}(q-1537\sqrt{q})>14=\sum_{i=1}^8 \deg(f_i(x))$,
hence for such $q$ it follows
from Theorem~\ref{thm-q14} and
Lemma~\ref{lem-kappa} that there exists $a\in\Fq$
such that $L_{a,a^2}$ is \maq.
For odd prime powers $9\le q<2.4\cdot 10^6$ with $q\equiv 1\pmod{4}$
the existence of \maq\ of order $q$ follows from Lemma~\ref{lem-24mil},
with the exceptions stated there. Since we have to deal with characteristic~$3$
separately anyway, the exceptions of Lemma~\ref{lem-24mil} do not bother us.

We will now deal with the characteristics $3$, $5$ and $23$.
By Lemma~\ref{lem-24mil}, \maqs\ of orders $p^2$ and $p^3$
exist for $p=3$ and $p=5$. Then the existence of \maq\ of orders $p^e$ 
for $e\ge 2$ follows from Lemma~\ref{lem-dir-product}.
Again by Lemma~\ref{lem-24mil}, \maq\ of order $23$ exists,
hence  there exist \maqs\ of order $23^e$ for all $e\ge 1$,
again by  Lemma~\ref{lem-dir-product}.
\end{proof}

\begin{theorem}
\label{thm-q3mod4}
Let $q\equiv 3\pmod{4}$ be a prime power, $q\ge 19$.
There exists a \maq\ of order $q$.
\end{theorem}
\begin{proof}
The structure of the proof is analogous to the proof of Theorem~\ref{thm-q1mod4}.
Instead of Theorem~\ref{thm-q14} we now use Theorem~\ref{thm-q34}.
The multiset of degrees of the polynomials $f_i$ is the same
as in the case $q\equiv 1\pmod{4}$, hence the Weil bound produces
the same results. Again we need to consider the special characteristics
in which the Weil bound may not apply.
Again 
by computing the discriminants of products $\prod_{i\in I} f_i(x)$
of even degree over $\Z$ we find that
the only primes that divide any of these discriminants
are $2$, $3$, $5$, $7$ and $23$.
Only the characteristic~$7$ is new and needs to be considered.
By Lemma~\ref{lem-24mil}, \maqs\ of orders $7^2$ and $7^3$
exist hence \maq\ of order $7^e$ exists for each $e\ge 2$.
\end{proof}

\section{Main Result}

For a positive integer $n$ and a prime $p$ let $\nu_p(n)$,
the {\em $p$-adic valuation} of $n$, be defined
as the largest integer $e$ such that $p^e$ divides~$n$.

\begin{theorem}
\label{thm-main}
Let $n$ be a positive integer such that $\nu_p(n)\neq 1$
for $p=3,5,7,11$
and $\nu_2(n)$ is even and different from $2$, $4$.
There exists a \maq\ of order~$n$.
\end{theorem}
\begin{proof}
Examples of order $2^6$ were constructed from Dickson nearfield
of that order \cite[Lemma~5.7]{DL}. Examples of orders $2^8$ 
and $2^{10}$ can be
obtained easily by prescribing
an automorphism group 
consisting of affine mappings $x\mapsto \alpha x+\beta$ 
where $\alpha$ is a non-zero cube.
The rest of the statement follows from Lemma~\ref{lem-dir-product}
and Theorems~\ref{thm-q1mod4} and \ref{thm-q3mod4}.
\end{proof}

\end{document}